\documentclass[twoside, 11pt]{article}

\usepackage{amssymb, amsmath, latexsym, mathrsfs, verbatim, calc}
\usepackage{cite}
\usepackage{amssymb}
\usepackage[usenames,dvipsnames]{color}
\usepackage{enumerate}
\usepackage{graphicx}
\usepackage[colorlinks, linkcolor=blue, anchorcolor=blue, citecolor=blue]{hyperref}

\numberwithin{equation}{section}
\newtheorem{thm}{Theorem}[section]

\newtheorem{prop}[thm]{Proposition}

\newtheorem{defn}[thm]{Definition}
\newtheorem{assum}[thm]{Assumption}

\definecolor{blue-violet}{rgb}{0.54, 0.17, 0.89}
\definecolor{purple1}{rgb}{0.63, 0.36, 0.94}
\definecolor{purple2}{rgb}{0.87, 0.0, 1.0}
\definecolor{purple}{rgb}{0.5, 0.0, 0.5}
\definecolor{pansypurple}{rgb}{0.47, 0.09, 0.29}
\definecolor{orange}{rgb}{1.0, 0.27, 0.0}

\def\ba{\begin{array}}
\def\ea{\end{array}}
\def\beq{\begin{equation}}
\def\endeq{\end{equation}}
\def\bes{\begin{equation*}}
\def\ees{\end{equation*}}
\def\bea{\begin{eqnarray}}
\def\eea{\end{eqnarray}}
\def\beaa{\begin{eqnarray*}}
\def\eeaa{\end{eqnarray*}}

\def\dis{\displaystyle}

\def\no{\noindent}

\def\lastline{\par \vspace{-7.3ex} \no}

\def\nts{\negthinspace}

\def\q{\quad}
\def\qq{\qquad}

\def\={=\nts \nts=\nts \nts=\nts \nts=}

\def\bh{\mbox}
\def\({\textnormal{(}}
\def\){\textnormal{)}}

\def\qx{\wedge}

\def\d{\delta}
\def\e{\varepsilon}

\def\k{\kappa}
\def\l{\lambda}

\def\t{\tau}

\def\th{\theta}

\def\vf{\varphi}

\def\ksfl{\begin{cases}}
\def\jsfl{\end{cases}}

\def\O{\Omega}

\def\P{\Psi}

\def\kspl{\begin{itemize}}
\def\jspl{\end{itemize}}

\def\bbh{\nonumber}

\def\bbh{\nonumber}

\def\no{\noindent}

\def\q{\quad}
\def\qq{\qquad}

\def\jf{\int}
\def\fs{\frac}
\def\yjt{\rightarrow}
\def\kspl{\begin{enumerate}[(1)]}
\def\jspl{\end{enumerate}}
\def\zb{\left}
\def\yb{\right}
\def\hx{\mathscr}
\def\qed{\hfill \rule[0cm]{.25cm}{.25cm}\medskip}   

\def\b1{{\bf 1}}

\newenvironment{proof}{\no {\bf Proof.\;}}{$\qed$\vspace{.1in}}

\newenvironment{itm}{\vspace{-1ex}\begin{itemize}}{\end{itemize}}
\def\bi{\begin{itm}}
\def\ei{\end{itm}}

\def\equ_ind{\arabic{section}.\arabic{equation}}
\def\sec_ind{\arabic{section}}

\usepackage{CJK}
\begin{document}
\begin{sloppypar}
\begin{CJK*}{GBK}{song}
\title{\Large \bf Minimizing the Ruin Probability under the Sparre Andersen Model
\thanks{This research is supported by Chinese NSF grants No. 11911530091, No. 11471171 and No. 11931018.}}

\author{{Linlin Tian}$^a$\footnote{E-mail:linlin.tian@mail.nankai.edu.cn}~~~
{Lihua Bai}$^a$\footnote{Corresponding author, E-mail:lhbai@nankai.edu.cn }~~~ {
}
\\
 \small  a. School of Mathematical Sciences, Nankai University, Tianjin 300071, China.}\,
\date{}
\maketitle
\begin{center}
\begin{minipage}{130mm}
{\bf Abstract.}
In this paper, we consider the problem of minimizing the ruin probability of an insurance company in which the surplus process follows the Sparre Andersen model.
Similar to Bai et al. \cite{bai2017optimal}, we recast this problem in a Markovian framework by adding another dimension representing the time elapsed since the last claim. After Markovization,  We investigate the regularity properties
of the value function, and state the dynamic programming principle. Furthermore,
we show that the value function is the unique  constrained viscosity
solution to the associated Hamilton-Jacobi-Bellman  equation. It should be noted that there is no discount factor in our paper, which makes it tricky to prove the uniqueness. To overcome this difficulty, we construct the strict viscosity supersolution. Then instead of comparing the usual viscosity supersolution and subsolution, we compare the supersolution and the strict subsolution. Eventually we show that all viscosity subsolution is less than the supersolution.

\vspace{3mm} {\bf Keywords:}  The Sparre Andersen model,  Minimizing the ruin probability, Optimal reinsurance policy, Viscosity solution.

\vspace{3mm} {\bf 2010 Mathematics Subject Classification}: 	49L25, 93E20, 	91B30
\end{minipage}
\end{center}

\section{Introduction}

We consider a minimizing ruin probability for an insurance company. This optimization problem was first suggested and studied by Cram$\acute{\mbox{e}}$r\cite{cramer1930mathematical}.
In the past twenty years or so, researchers have used the minimizing  ruin probability as a criterion for dynamically determining the optimal investment and reinsurance policy; see, for example, Azcue and Muler \cite{azcue2013minimizing} study the minimizing ruin probability problem assuming the management can invest dynamically part of the reserve in the non-cash asset. Liang and Young \cite{Liang2018Minimizing} considered the optimal investment and reinsurance strategy for an insurance company when the risk process follows a compound Poisson process. For more introduction of minimizing ruin probability, see  Gajek and Zagrodny \cite{Gajek2004Reinsurance},
Hipp and Plum \cite{HippOptimal},
Hipp and Taksar  \cite{hipp2010optimal},
 Meng and Zhang  \cite{meng2010optimal},
 Li and  Young \cite{danping}.

 For an insurance company, buying reinsurance to lower the claim risk is a natural choice. In this case,
an insurance company can minimize the ruin probability by finding the optimal reinsurance policy. In our model,
we study the finite-time minimizing the ruin probability of a compound renewal model, which has several distinct features in contrast to the existing literature since the wealth process is non-Markovian. We aim to maximize the survival probability by controlling the reinsurance retention level.
For the Markov process, one can explore the optimization problem by the stochastic optimal control theory. In the application of stochastic optimal control theory, one can associate a Hamilton-Jacobi-Bellman (HJB) equation to the stochastic optimal control problem by the dynamic programming principle (DPP) approach. But in our paper, the reserve process follows the Sparre Andersen model, which is no longer Markovian. Similar to  Bai et al.\cite{bai2017optimal}, we plan to ``Markovize" the model first, i.e., we apply the so-called $Backward$ $Markovization$ $technique$ (cf., e.g., \cite{rolski2009stochastic}). After Markovization,  we can study this optimization problem via the DPP approach. Specifically, we shall first investigate the regularity properties of the value function and then state the DPP, from which we can formally derive the associated HJB equation to which the value function is a solution in some sense.

The HJB equation associated to our problem is an equation involving a first-order integro-differential operator.
Since it is hard to conjecture the existence of continuously differentiable solutions for our HJB equation, it is natural to invoke the notion of viscosity solution as done by Azcue and Muler \cite{azcue2013minimizing}.
We recall that the notion of viscosity solutions was introduced by Crandall and Lions \cite{crandall1983viscosity} for the first-order equations and Lions \cite{lions1982optimal,lions1983optimal}
for the second-order equations.
It merely requires the continuity of the value function to define the viscosity solution. We refer to the
user's guide of Crandall, Ishii and Lions \cite{crandall1992user} and the lecture notes in Bardi et al. \cite{bardi2006viscosity}
for an overview of the viscosity solutions theory and its applications.

For our problem,  we can not establish all the explicit boundary condition for the value function based on the information of the optimization problem. The lack of boundary conditions of the HJB equation makes it impossible to prove the uniqueness of the solution. To overcome this difficulty,  we need to invoke the notion of  $constrained$ $viscosity$ $solution$ (see, e.g., Soner \cite{SonerOptimal} and Bai et al. \cite{bai2017optimal}), and as it turns out we can  show that the value function is indeed a constrained viscosity solution of the HJB equation on an appropriately defined domain. To the end, we show that the value function is the unique solution of the associated HJB equation.

When we are proving the uniqueness of the viscosity solutions, the main difficulty is that there is no discount factor in this model, or in other words, the coefficient of the function $V$  is $0$ in the HJB equation. As we can see,  other optimization papers' uniqueness proofs rely on the discount factor being positive,  see, e.g.,\cite{azcue2005,bai2017optimal,benth}. In our paper, we overcome the difficulty of lacking discount factors by  constructing the strict viscosity supersolution.
For a given supersolution, we construct a strict supersolution,
then, instead of comparing the usual supersolution and subsolution, we compare the size of subsolution and strict supersolution. Eventually, we can show the comparison holds among all viscosity subsolution and supersolution.

The rest of the paper is organized as follows. In section 2, we establish the basic setting and assumptions. In section 3, we study the properties of the value function and prove the continuity of the value function in the temporal variable. In section 4, we state the DPP and show that the value function is a constrained viscosity solution to the associated  HJB equation. Finally, in Section 5, we prove the comparison principle, hence prove that
the value function is the unique constrained viscosity solution of the corresponding HJB equation.

\section{Model and Assumption}

Throughout this paper, we work with a complete filtered probability space $(\O,\hx{F},\{\hx{F}_t\},\mathbb{P})$ on which is defined a renewal counting process $N=\{N_t\}_{t\ge 0}$. For this counting process $N_t$,
we denote $\{\sigma_n\}_{n=1}^\infty$ be the jump times ($\sigma_0:=0$) and $T_i=\sigma_i-\sigma_{i-1}$, $i=1,2,\ldots$ to be the time elapses between successive jumps. We assume that $T_i$'$s$ are independent and identically distributed with a common distribution $F:\mathbb{R}_+\mapsto \mathbb{R}_+$ and there exists an intensity function $\lambda:[0,+\infty)\mapsto[0,+\infty)$ such that $\bar{F}(t)=\mathbb{P}(T_1>t)=\exp\{-\int_0^t\lambda(u)du\}$.

Let $T>0$ be a given time horizon. Let  $N_t$ be a renewal counting process we mentioned before representing the frequency of the incoming claims and $\{U_i\}_{i=1}^{\infty}$ a sequence of random variables representing
the ``size" of the incoming claims. We assume that $\{U_i\}$ are i.i.d. with a
common distribution $G:\mathbb{R}_+\mapsto \mathbb{R}_+$, independent of $N$. Denote $Q_t:=\sum_{i=1}^{N_t}U_i$ for simplicity. Since $Q$ is non-Markovian in general (unless the counting process $N$ is a Poisson process), we cannot apply the dynamic programming principle directly. Therefore, we apply the  Backward Markovization technique to ``Markovize" $Q_t$ first. In other words, we  define a new process $W_t:=t-\sigma_{N_t}, t\ge 0$, representing the time elapsed since the last claim. It is known that $(t, Q_t, W_t)$ is a piecewise deterministic Markov process, see e.g., \cite{rolski2009stochastic}). We note that $0\le W_t\le t\le T$, for $t\in[0,T].$
Throughout this paper, we consider the filtration $\{\mathscr{F}\}_{t\ge 0}$, in which $\mathscr{F}_t:=\mathscr{F}_t^Q \vee\mathscr{F}_t^W$, $t\ge 0.$ Here $\{\mathscr{F}_t^{\xi}:t\ge 0\}$ denotes the natural filtration generated by the process $\xi=Q, W$, respectively, with the usual $\mathbb{P}$-augmentation such that it satisfies the $usual$ $hypotheses$ (cf., e.g., Propter \cite{propter}).

After Markovization, we can apply the dynamic optimal control theory, which means we can start at any time $t\in [0,T]$. In other words, instead of starting the clock at $0$, we start from $s\in[0,T]$ such that $W_s=w,\mathbb{P}$-a.s. Under the regular conditional probability distribution $\mathbb{P}_{sw}(\cdot):=\mathbb{P}(\cdot|W_s=w)$ on $(\Omega,\mathscr{F})$, we consider the ``shifted" version of processes $(Q,W)$ on the space $(\Omega, \mathscr{F}, \mathbb{P}_{sw}; \{\mathscr{F}_{t}\}_{t\ge s})$. We define a new counting process $N_t^s:=N_t-N_s$ starting at time $s\in[0,T]$, where $ t\in[s,T]$. Then $N^s$ is a ``delayed" renewal process. At the same time, its waiting times $T_i^s, i\ge 2$, remain independent, identically distributed as the
original $T_i's$. Denote $T_1^{s,w}:=T_{N_s+1}-w=\sigma_{N_s+1}-s$ the ``time-to-first-jump" and $T_1^{s,w}$ follows the following probability
\begin{equation}\nonumber
\mathbb{P}_{sw}(T_1^{s,w}>t)=\mathbb{P}(T_1>t+w|T_1>w)=e^{-\int_w^{w+t}\lambda(u)du}.
\end{equation}
In the following, we denote $N_t^{s,w}:=N_t^{s,w}|_{W_s=w}, Q_t^{s,w}:=\sum_{i=1}^{N_t^{s,w}}U_i$ and $W_t^{s,w}:=w+W_t-W_s, t\ge s$ for simplicity. It is seen that $(Q_t^{s,w}, W_t^{s,w}), t\ge s$ is a Markov $\mathscr{F}_t$-adapted process defined on $(\Omega, \mathscr{F}, \mathbb{P}_{sw}).$

Now we introduce the insurance model.
In this paper, we assume that the dynamics of surplus of an insurance company, denoted by $X=\{X_t\}_{t\ge 0}$, in the absence of reinsurance, follows the Sparre Andersen model:
\begin{equation}
{X}_t:=x+pt-Q_t=x+pt-\sum_{i=1}^{N_t}U_i,\quad\quad t\in[0,T],
\end{equation}
where $x=X_0\ge 0$, $p>0$ is a premium rate, $N_t$ is the  pre-mentioned renewal counting process  representing the frequency
of the incoming claims.
We define the control process by $\pi_t$, $t\ge 0$, where $\pi\in[0,1]$ representing the risk exposure, which means, for a fixed $\pi$, $100\pi\%$ of each claim is paid by the insurance company while $100(1-\pi)\%$ is paid by the reinsurer. Then, $p(1+\eta)(1-\pi)$ is the rate at which the premiums are diverted to the reinsurer by the insurance company, where $\eta>0$ is the safety loading of the reinsurance company. Notice that the reinsurance is called cheap when $\eta=0$ and non-cheap when $\eta>0.$

Throughout this paper, we will consider the the filtration $\mathscr{F}=\mathscr{F}^{(Q,W)}$ and we say that a control strategy  $\pi=\{\pi_t\}_{t\ge 0}$ is admissible if it is $\mathscr{F}$-predictable with c$\grave{\mbox{a}}$dl$\grave{\mbox{a}}$g paths, and square-integrable (i.e., $\mathbb{E}[\int_0^T|\pi_t|^2dt]<+\infty$). We denote the set of all admissible strategies restricted to $[s,T]\subseteq[0,T]$ by $U_{ad}^{s,x,w}[s,T]$. Notice that we labeled the $w$ on the upper right corner to emphasize the dependence on $w$.
For
$\pi\in U_{ad}^{s,x,w}[s,T]$ and initial surplus $x$, the dynamics of the controlled risk process $X_t^{\pi,s,x,w}$ satisfies the following stochastic differential equation (SDE):
\bea\label{model}
dX_t^{\pi,s,x,w}=[\pi_t(1+\eta)-\eta]pdt-\pi_td\sum_{i=N_s+1}^{N_t^{s,w}}U_i.
\eea
The process of the time elapsed since the last claim $W_t$ follows
\begin{equation}
W_t=w+(t-s)-(\sigma_{N_t}-\sigma_{N_s}).
\end{equation}
For any $\pi\in U_{ad}^{s,x,w}[s,T]$ , we denote $\t^\pi=\tau^{\pi,s,x,w}:=\inf\{t>s:X_t^{\pi,s,x,w}<0\}$ to be the ruin time of the insurance company. We shall make use of the following standing assumptions.
\begin{assum}
\label{ass1}
The insurance
premium $p$ is a positive constant.
The distribution function $G$ (of $U_i's$) is continuous on
$[0,+\infty)$. The distribution function $F$ (of $T_i's$)  is absolutely continuous, with density function $f$ and intensity
function $\lambda(t):={f(t)}\slash{\bar{F}(t)}>0,t\in[0,T]$.
\end{assum}
We now describe our optimization problem. Given an admissible strategy $\{\pi_t\}_{t\in[s,T]}\in U_{ad}^{s,x,w}[s,T]$,
 the corresponding survival probability of strategy $\pi$ is denoted by $J^\pi(s,x,w;\pi)$, in other words,
\begin{align}\nonumber
 J(s,x,w;\pi)&=\mathbb{P}(\tau^{\pi,s,x,w}\ge T|X_s^\pi=x,W_s=w)\\&=\mathbb{P}(X_t^{\pi,s,x,w}\ge 0,\; \mbox{for all}\; t\in[s,T]).
\end{align}
We aim to maximize the survival probability. Now we define the value function as the supremum of the survival probability, which means
\bea\label{opti1}
V(s,x,w):=\sup_{\pi\in U_{ad}^{s,x,w}[s,T]} J(s,x,w;\pi).
\eea
Noticing for all $x>\eta p(T-s)$, the survival probability on $[s,T]$ is $1$. Thus, the survival probability and the value function that we are about to study should be defined on  \[D:=\{(s,x,w):0\le s\le T,0\le  x\le\eta p(T-s) , 0\le w\le s\}.\]

We shall frequently carry out our discussion on the following two sets:
\begin{align}\nonumber
&\mathscr{D}:=int D=\{(s,x,w)\in D: 0<s<T, 0<x<\eta p(T-s),0<w<s \},\\\nonumber
&\mathscr{D}^*:=\{(s,x,w)\in D: 0\le s<T, 0\le x<\eta p(T-s),0\le w \le s\}.
\end{align}

\section{Basic Properties of the Value Function}
\begin{prop}\label{conx}
Assume the Assumption \ref{ass1} is in force, then for all $(s,x,w), (s+h,x,w)\in D,h>0$, \\
(1) the value function $V$ satisfies $V(s,x,w)\le V(s+h,x,w).$\\
(2) the value function $V$ is continuous with respect to $s$, uniformly for $s,x,w$.
\end{prop}
The Proposition \ref{conx}-(1) is obvious. The proof of the Proposition \ref{conx}-(2) is similar with that of Proposition 3.3 of Bai et al. \cite{bai2017optimal}, we omit it for the sake of brevity.
\begin{prop}\label{prop2}
Assume the Assumption \ref{ass1} is in force,  then the value function $V$ enjoys the following properties: \\
(1) For all $x_1\le  x_2$, $(s,x_1,w), (s,x_2,w)\in D$, $V(s,x_1,w)\le V(s,x_2,w)$.\\
(2) $V$ is continuous with respect to $x$.
\end{prop}
\begin{proof}\q (1) The first claim is obviously true. 

(2) Suppose that $x_2-x_1=h>0$.
For any strategy $\pi_2\in U_{ad}^{s,x_2,w}[s,T]$, define strategy $\pi_1(t):=\pi_2(t)$ for all  $t\in [s,T]$. Apparently,  $\pi_1\in U_{ad}^{s,x_2,w}[s,T]$.
Denote the reserve processes by $X_t^{\pi_1,s,x_1,w}$, $X_t^{\pi_2,s,x_2,w}$, the ruin time by $\tau_1$, $\tau_2$ of strategies $\pi_1, \pi_2$, respectively.
Denote $Y_t:=X_t^{\pi_2,s,x_2,w}-X_t^{\pi_1,s,x_1,w}$. Notice that  $Y_t=h$ for  all $s\le t<\tau_1\wedge T.$ We can see that
\begin{align}\bbh
&J(s,x_2,w;\pi_2)-J(s,x_1,w;\pi_1)\\\bbh
=& \mathbb{P}(\mbox{for all}\; t\in [s,T], X_t^{\pi_2,s, x_2,w}\ge 0 )-\mathbb{P}(\mbox{for all}\; t\in [s,T], X_t^{\pi_1,s,x_1,w}\ge 0)\\\nonumber
\le&\mathbb{P}(X_{\t_1-}^{\pi_1,s,x_1,w}<\Delta X_{\t_1}^{\pi_1,s,x_1,w}\le X_{\t_1-}^{\pi_1,s,x_1,w}+Y_{\t_1},\t_1\in[s,T] )\le\mathbb{E}(G(X_{\t_1-}^{\pi_1,s,x_1,w}+h)-G(X_{\t_1-}^{\pi_1,s,x_1,w})).
\end{align}
Since $G$ is uniformly continuous, we see that for all $\e>0$, there exists a constant $\d>0$ (irrelevant with $s,x,w$) such that for all $h<\d$ and $x\ge 0$, $G(x+h)-G(x)\le \e.$
Thus, we see that for all $h\in(0, \d],$
\[J(s,x_1+h,w;\pi_2)-J(s,x_1,w;\pi_1)\le\e.\]
Since $\pi_2\in U_{ad}^{s,x_2,w}[s,T] $ is arbitrary, we see that for all $0<h<\delta$, $V(s,x+h,w)-V(s,x,w)\le \e.$ Combing with Proposition \ref{prop2}-(1),the proof of the continuity about $x$ is completed. \end{proof}
\begin{prop}\label{aboutw1} Assume the Assumption \ref{ass1} is in force, then for all $0\le s<s+h<T, (s,x,w), (s+h,x,w)\in D$,
the value function $V$ satisfies the following properties:
\begin{enumerate}[(1)]
  \item  \bea
\label{conw}
V(s,x,w)\ge \exp\zb\{-\jf_w^{w+h}\l(u)du\yb\}V(s+h,x,w+h).
\eea
  \item $V$ is continuous with respect to $w$, uniformly for $(s,x,w)$ in $D$.
\end{enumerate}
\end{prop}
\begin{proof}
 (1) For the initial data $(s,x,w)$, define strategy $\hat{\pi}$ as follows:
 \[\hat{\pi}_t=\mathbf{1}_{\{T_1^{s,w}\ge h\}}\left(\mathbf{1}_{[s,s+h)}(t)+\mathbf{1}_{[s+h,T]}(t)\pi(t)\right)+\mathbf{1}_{\{T_1^{s,w}<h\}}\mathbf{1}_{[s,T]}(t).\]
where $\pi$ denotes any admissible strategy on $[s+h,T]$ and $\mathbf{1}_{A}$ is the indicator function of set $A$.
Notice that on  $\{T_1^{s,w}\ge h\}$, at time $s+h$, the time  elapsed since the last claim $W_{s+h}=w+h$ and the surplus $X_{s+h}^{\hat{\pi},s,x,w}\ge x$.
Thus, we see that
\[
V(s,x,w)\ge J(s,x,w;\hat{\pi})\ge e^{-\jf_w^{w+h}\l(u)du}J\zb(s+h,x,w+h;\pi\yb).
\]
Since $\pi$ is arbitrary, we see that
\[
V(s,x,w)\ge  e^{-\jf_w^{w+h}\l(u)du}V\zb(s+h,x,w+h\yb).
\]
(2) First, combing Proposition \ref{conx}-(1) with Proposition \ref{aboutw1}-(1), we can directly calculate as follows:
\begin{align}\bbh
&V(s,x,w)-V(s,x,w+h)\\\bbh
=&V(s,x,w)-V(s+h,x,w+h)+V(s+h,x,w+h)-V(s,x,w+h)\\\bbh
 \ge&\zb(\exp\zb\{-\jf_w^{w+h}\l(u)du\yb\}-1\yb)V(s+h,x,w+h)
 \ge\zb(\exp\zb\{-\jf_w^{w+h}\l(u)du\yb\}-1\yb).
\end{align}
Letting $h\downarrow 0,$ we see that
\[\varliminf\limits_{{h\downarrow 0}}(V(s,x,w)-V(s,x,w+h))\ge 0.\]
We only need to consider the other direction. Since we can prove $\varlimsup\limits_{h\downarrow 0}V(s,x,w)-V(s,x,w+h)\le 0$ by the similar idea which is used in Tian et  al. \cite{tian}, we omit the detailed proof here.
Until now,    the proof of the continuity of $V$ with respect to $w$ is completed. \end{proof}
\section{The Hamilton-Jacobi-Bellman equation}
\begin{thm}
(Dynamic programming principle)
Assume that Assumption (\ref{ass1}) is in force. Then, for any $(s,x,w)\in D$ and for any stopping time $\t\in[s,T]$, it holds that
\bea\label{dnm1}
V(s,x,w)=\sup_{\pi\in U_{ad}^{s,x,w}[s,T]}\mathbb{E}_{sxw}\zb[V(\t\qx\t^\pi, X_{\t\qx\t^\pi}^{\pi,s,x,w}, W_{\t\qx\t^\pi})\yb].
\eea
\end{thm}
Similar to Bai et al. \cite{bai2017optimal}, one can show that the value function $V$ fulfills the dynamic programming principle (DPP). For brevity's sake, we omit the proof here.

Now we are ready to investigate the main subject of the paper: the Hamilton-Jacobi-Bellman (HJB) equation associated to our optimization problem (\ref{opti1}). The main content of this section is to show that the value function $V$ is a viscosity solution of the following HJB equation:
\bea\label{hjb}
\begin{cases}
\max_{q\in[0,1]}\bigg\{ p[q(1+\eta)-\eta]V_x+V_s+V_w\\
\qq\qq\qq\;+\l(w)\jf_0^{\fs{x}{q}}V(s,x-q y,0)dG(y)-\l(w)V(s,x,w)\bigg\}=0,\q (s,x,w)\in\mathscr{D};\\
V(T,x,w)=1,x\ge 0;\\
V(t,x,w)=1, x\ge \eta p(T-t).
\end{cases}
\eea
 Denote  $\mathbb{C}^{1,1,1}(D)$ the set of all continuously differentiable  functions on $D$. Define the  first-order integro-differential operator for $\varphi\in \mathbb{C}^{1,1,1}(D)$:
\begin{align}\bbh
&\hx{L}[\varphi](s,x,w)\\\nonumber:=&\max_{q\in[0,1]}\zb\{ p[q(1+\eta)-\eta]\vf_x+\vf_s+\vf_w+\l(w)\jf_0^{\frac{x}{q}}\vf(s,x- q y)dG(y)-\l(w)\vf(s,x,w)\yb\}.
\end{align}
The HJB equation (\ref{hjb}) is well defined for all $u\in \mathbb{C}^{1,1,1}(\overline{\mathscr{D}^*}).$ However, in many applications the value function defined in (\ref{opti1}) is not continuously differentiable and the HJB equation should be interpreted in a weaker sense, which means,  we need to study the viscosity solutions of HJB equation. The precise definition of viscosity solution goes as follows:
\begin{defn}
Let $\mathscr{O}\subseteq\hx{D}^*$ be a subset such that  $\partial_T \mathscr{O}:=\{(T,y,v)\in \partial \mathscr{O}\}\neq \emptyset$, where $\overline{\mathscr{O}}$ is the closure of $\mathscr{O}$. Denote $\mathbb{C}(\mathscr{O})$ as the set of all continuous functions on $\mathscr{O}$.

(a) Let $v\in \mathbb{{C}}(\mathscr{O})$; we call $v$ a viscosity subsolution of (\ref{hjb}) on $\mathscr{O}$ if $v(T,y,v)\le 1$ for $(T,y,v)\in \partial_T \mathscr{O}$; $v(t,x,w)\le1$, for $x\ge \eta p(T-t)$ and for any $(s,x,w)\in \mathscr{O}$, $\vf \in\mathbb{C}^{1,1,1}(\bar{\mathscr{O}})$ such that $[v-\vf](s,x,w)=\max_{(t,y,v)\in \mathscr{O}}[v-\vf](t,y,v)$, it holds that
\[
\hx{L}[\vf](s,x,w)\ge 0.
\]

(b) Let $v\in \mathbb{{C}}(\mathscr{O})$; we call $v$ a viscosity supersolution of (\ref{hjb}) on $\mathscr{O}$ if $v(T,y,v)\ge 1$ for all $(T,y,v)\in \partial_T \mathscr{O} $;  $v(t,x,w)\ge1$, for $x\ge \eta p(T-t)$ and  for any $(s,x,w)\in \mathscr{O}$,  $\vf \in\mathbb{C}^{1,1,1}(\bar{\mathscr{O}})$ such that $0=[v-\vf](s,x,w)=\min_{(t,y,v)\in \mathscr{O}}[v-\vf](t,y,v)$, it holds that
\[
\hx{L}[\vf](s,x,w)\le 0.
\]
In particular, we call $u$ a ``constrained viscosity solution'' of (\ref{hjb}) on $\hx{D}^*$ if it is both a viscosity subsolution on $\hx{D}^*$ and a viscosity supersolution on $\hx{D}.$
\end{defn}
We now have an equivalent formulation of viscosity solution. The proof of the equivalence of two definitions is standard (e.g., see Benth et al. \cite{benth} and Awatif \cite{Awatif}). In this paper, we use both definitions interchangeably. Now, we introduce the alternative definition of viscosity solution. Given a continuously differentiable function $\varphi$ and a continuous  function $u$, we define the operator
\begin{align}\bbh
&\hx{L}[u,\varphi](s,x,w)\\\nonumber:=&\max_{q\in[0,1]}\zb\{ p[q(1+\eta)-\eta]\vf_x+\vf_s+\vf_w+\l(w)\int_0^{\frac{x}{q}}u(s,x- q y)dG(y)-\l(w)u(s,x,w)\yb\}.
\end{align}
\begin{defn}
Let $v\in {\mathbb{C}}(\mathscr{O})$; we call $v$ a viscosity subsolution of (\ref{hjb}) on $\mathscr{O}$ if $v(T,y,v)\le 1$ for $(T,y,v)\in \partial_T \mathscr{O}$;  $v(t,x,w)\le1,$ for $ x\ge \eta p(T-t)$ and for any $(s,x,w)\in \mathscr{O}$, $\vf \in\mathbb{C}^{1,1,1}(\overline{\mathscr{O}})$ such that $0=[v-\vf](s,x,w)=\max_{(t,y,v)\in \mathscr{O}}[v-\vf](t,y,v)$, it holds that
\[
\hx{L}[v,\vf](s,x,w)\ge 0.
\]

Let $v\in \mathbb{{C}}(\mathscr{O})$; we call $v$ a viscosity supersolution of (\ref{hjb}) on $\mathscr{O}$ if $v(T,y,v)\ge 1$ for all $(T,y,v)\in \partial_T \mathscr{O} $;  $v(t,x,w)\ge1$, for $x\ge \eta p(T-t)$ and for any $(s,x,w)\in \mathscr{O}$, $\vf \in\mathbb{C}^{1,1,1}(\overline{\mathscr{O}})$ such that $0=[v-\vf](s,x,w)=\min_{(t,y,v)\in \mathscr{O}}[v-\vf](t,y,v)$, it holds that
\[
\hx{L}[v,\vf](s,x,w)\le 0.
\] 
\end{defn}
\begin{thm}
The value function is a constrained viscosity solution of (\ref{hjb}) on $\hx{D}^*$.
\end{thm}
\begin{proof}
 $supersolution$ Given $(s,x,w)\in\hx{D}$.  Let $\vf\in \mathbb{C}^{1,1,1}(D)$ such that $V-\vf$ attains its minimum at $(s,x,w)$ with $V(s,x,w)=\vf(s,x,w)$. Consider the strategy $\pi^0$ with the reinsurance rate $q_0$, where $q_0\in [0,1]$; $T_1^{s,w}$ denotes the time of the first claim. Take $h>0$ such that $h<\fs{x}{p\eta}$.  Denote $\tau_s^h:= s+h\wedge T_1^{s,w}$ and $R_t^{s,x,w}:=(t,X_t^{\pi^0,s,x,w},W_t^{s,w})$. By the dynamic programming principle, we have
\begin{align}\bbh
V(s,x,w)\ge& \mathbb{E}\zb[V(\t_s^{h}, X_{\t_s^{h}}^{s,x,w}, W_{\t_s^{h}})\yb]\\\bbh
=&\mathbb{E}_{sxw}\left\{\left[V(R_{\tau_s^h}^{s,x,w})-V (R_{{\tau_s^h}-}^{s,x,w})\right]\mathbf{1}_{\{T_1^{s,w}<h\}}\right\}
+\mathbb{E}_{sxw}\left[V (R_{{\tau_s^h}-}^{s,x,w})\right].
\end{align}
Using the fact that $V-\varphi$ attains its minimum at $(s,x,w)$, we obtain
\begin{align}\label{sup1}
0\ge \mathbb{E}_{sxw}\left\{\left[V(R_{\tau_s^h}^{s,x,w})-V (R_{{\tau_s^h}-}^{s,x,w})\right]\mathbf{1}_{\{T_1^{s,w}<h\}}\right\}
+\mathbb{E}_{sxw}\left[\varphi (R_{{\tau_s^h}-}^{s,x,w})-\vf(s,x,w)\right] :=I_1+I_2,
\end{align}
where $I_1$ and $I_2$ are the two terms on the right-hand side above. Since $\t_s^h=s+T_1^{s,w}$ on $\{T_1^{s,w}< h\}$, we have
\begin{align}\bbh
I_1=&\mathbb{E}_{sxw}\left\{\left[V(R_{s+T_1^{s,w}}^{s,x,w})-V (R_{{s+T_1^{s,w}}-}^{s,x,w})\right]\mathbf{1}_{\{T_1^{s,w}<h\}}\right\}\\\label{I1}
=&
\mathbb{E}_{sxw}\!\zb[ \int_0^h\!\int_0^\infty e^{-ct}\left[V (s+t,X_{s+t-}^{\pi^0,s,x,w}\!-q_0u,0)\!-\!V(s+t,X_{{s+t}-}^{\pi^0,s,x,w},W_{{s+t}-}^{s,w})\right]dG({u})dF_{T_1^{s,w}}(t)\yb].
\end{align}
As there are no jumps on $[s,\tau_s^h)$, using It$\hat{\bh{o}}$'s formula, we obtain that
\begin{align}\nonumber
I_2&=\mathbb{E}_{sxw}\left[\int_{s}^{\tau_s^h}\left[\varphi_t+p[q_0(1+\eta)\varphi_x-\eta]+\varphi_w\right](R_u^0)du\right]\\\label{I2}&= \mathbb{E}_{sxw}\zb[\int_{s}^{s+h}\bar{F}_{T_1^{s,w}}(u-s)\left[\varphi_t+\varphi_xp[q_0(1+\eta)-\eta]+\varphi_w\right](R_u^0)du\yb].
\end{align}
Recall that $F_{T_1^{s,w}}(t)=1-e^{-\jf_w^{w+t}\l(u)du}$. Dividing both sides of (\ref{sup1}) and then letting $h\downarrow0$, due to (\ref{I1}), (\ref{I2}) and $[V-\vf](s,x,w)=0$, we obtain
\bea
 p[q_0(1+\eta)-\eta]\vf_x+\vf_s+\vf_w+\l(w)\jf_0^{\frac{x}{q_{0}}}\vf(s,x- q_0y,0)dG({y})-\l(w)\vf(s,x,w)\le 0.
\eea
Since $q_0\in[0,1]$ is arbitrary, we see that
\[\hx{L}[\vf](s,x,w)\le 0.\]
Now we complete the proof of the value function being a viscosity supersolution of the HJB equation.

$subsolution$ Now we show that $V$ is a viscosity susolution of HJB equation on $\mathscr{D}^*$.  If $V$ is not a viscosity subsolution on $\hx{D}^*$, then there exists a point $(s,x,w)\in \hx{D}^*$ and a $\psi^0\in \mathbb{C}^{1,1,1}(D)$ such that $0=[V-\psi^0](s,x,w)=\max_{(t,y,v)\in \hx{D}^*}[V-\psi^0](t,y,v)$, but
 \begin{equation}\nonumber
 \hx{L}[\psi^0](s,x,w)=-2\zeta<0,
\end{equation}
where $\zeta>0$ is a constant.
Fix the strategy $\pi\in U_{ad}^{s,x,w}[s,T]$ and let $R_t^{s,x,w}=(t,X_t^{s,x,w}, W_t^{s,w})$. Define $\tau_{\rho}:=\inf\{t>s: R_t\notin\overline{B_{\rho}(s,x,w)\cap\mathscr{D}^*}\}$, where $B_{\rho}(s,x,w)$ is the open ball centered at $(s,x,w)$ with radius $\rho$.  Since $s\le \tau_{\rho}\le s+\rho$, when $\rho\yjt 0$,  $\tau_{\rho}\rightarrow s$. Thus, for any given $\e_1>0$ there exists a constant $\rho_1>0$ such that for all $\rho\in(0,\rho_1]$, $\tau_\rho-s<\e_1$, thus,
 \begin{equation}\label{subsolution1}
 \mathbb{P}(T_1^{s,w}>\varepsilon_1)<\mathbb{P}(T_1^{s,w}>\tau_{\rho}-s).
 \end{equation}
Now we claim that there exist constants $\rho\in(0,\rho_1), \e\in(0,+\infty)$ and a function $\psi\in \mathbb{C}^{1,1,1}(D)$ such that
\begin{align}\label{subsolution2}
\mathscr{L}[\psi](s,x,w)&\le -\varepsilon , (t,y,v)\in \overline{B_{\rho}(s,x,w)\cap\hx{D}^*}\backslash\{t=T\}\cup\{x=\eta p(T-t)\};\\\label{subsolution3}
V(t,y,v)&\le \psi(t,y,v)-\varepsilon,  (t,y,v)\in \partial B_{\rho}(s,x,w)\cap\mathscr{D}^*.
\end{align}
 To see this, we consider two cases.

$Case\; 1 $ $x>0$. In this case, we introduce the function
 \begin{equation}
 \psi(t,y,v)=\psi^0(t,y,v)+\frac{\zeta[(t-s)^2+(y-x)^2+(v-w)^2]^2}{\lambda(w)(x^2+w^2)^2}, \quad (t,y,v)\in D.
\end{equation}
Then 
 $\mathscr{L}[\psi]<-\zeta<0$.
 By the continuity of $\hx{L}[\psi]$, we can find a positive constant $\rho<\rho_1$ such that
\begin{equation}
\hx{L}[\psi](t,y,v)<-\frac{\zeta}{2}, (t,y,v)\in \overline{B_{\rho}(s,x,w)\cap\hx{D}^*}\backslash\{t=T\}\cup\{x=\eta p(T-t)\}.
 \end{equation}
 Note that for $(t,y,v)\in \partial B_{\rho}(s,x,w)\cap \hx{D}^*$, one has
 \begin{equation}
V(t,y,v)\le \psi(t,y,v)-\frac{\zeta \rho^4}{\lambda(w)(x^2+w^2)^2}.
 \end{equation}
By choosing $\e:=\min\{\frac{\zeta}{2},\frac{\zeta \rho^4}{\lambda(w)(x^2+w^2)^2}\}$ we obtain (\ref{subsolution2}), (\ref{subsolution3}).

$Case\; 2 $ $x=0$. In this case, we introduce the function
\begin{equation}
\psi(t,y,v)=\psi^0(t,y,v)+\zeta\left[(t-s)^2+y^2+(v-w)^2\right], (t,y,v)\in D.
\end{equation}
In fact,  at the point $(s,0,w)$, $\hx{L}[\psi](s,0,w)=\hx{L}[\psi](s,0,w)=-2\zeta<0$. Thus, there exists a positive constant $\rho<\rho_1$ such that
$\hx{L}[\psi](t,y,v)<-\zeta<0$ on $\overline{B_\rho(s,0,w)\cap\hx{D}^*}\backslash(\{t=T\}\cup\{x=\eta p(T-t)\})$. If we define $\e:=\min\{{\zeta},\zeta\rho^2\}$, then a similar calculation as before shows that  (\ref{subsolution2}), (\ref{subsolution3}) still
holds, proving the claim.

We now argue that this claim leads to a contradiction. Define $\tau:=\tau_{\rho}\wedge T_1^{s,w}$. Applying It$\hat{\mbox{o}}$
formula we obtain that
\begin{align}\nonumber
\mathbb{E}_{sxw}\left[V(R_\tau^{s,x,w})\right]&=\mathbb{E}_{sxw}\left[\psi(R_{\tau}^{s,x,w})+V(R_{\tau}^{s,x,w})-\psi(R_{\tau}^{s,x,w})\right]\\
&\le \psi(s,x,w)+\mathbb{E}_{sxw}\left[\int_s^\tau\mathscr{L}[\psi](R_t^{s,x,w})dt-\varepsilon\mathbf{1}_{\{\tau_\rho<T_1^{s,w}+s\}}\right].
\end{align}
Since $\mathscr{L}[\psi]\le -\varepsilon$ on $[s,\tau)$, we can get that that
\begin{align}\label{b1}
\mathbb{E}_{sxw}\left[V(R_\tau^{s,x,w})\right]\le \psi(s,x,w)-\e \mathbb{P}(\tau_\rho<T_1^{s,w}+s)=V(s,x,w)-\e \mathbb{P}(\tau_\rho<T_1^{s,w}+s).
\end{align}
By (\ref{subsolution1}),
$
\mathbb{P}(T_1^{s,w}+s>\tau_{\rho})> \mathbb{P}(T_1^{s,w}>\varepsilon_1)>0.
$
 Thus, we can see that (\ref{b1}) contradicts the dynamic programming principle (\ref{dnm1}). Now we show that the value function is a constrained viscosity solution of HJB equation on $\mathscr{D}^*$. \end{proof}
\section{Uniqueness}
In  this section, we present a comparison theorem that would imply the uniqueness among all constrained viscosity solutions.
In our model, there is no discount factor which makes the proof of uniqueness more tricky. Luckily, inspired by Mou and \'{S}wi\c{e}ch  \cite{Mou2015Uniqueness}, we can overcome this difficulty by constructing a strict viscosity supersolution for the HJB equation.
\begin{thm} Assume that Assumption \ref{ass1} is in force. Let $u$ be a viscosity subsolution of (\ref{hjb}) on $\mathscr{D}^*$ and $v$ be a viscosity supersolution
of (\ref{hjb}) on $\mathscr{D}$, then $u\le v$ on $D$.
\end{thm}
\begin{proof}
First, we define $\vf^{\varsigma,\th}(t,y,v):=\frac{\varsigma}{t+1}+\frac{\th(T-t)}{t}$, where $\th>0$ and $\varsigma>0$ are two constants. Then it is straightforward to check that $v+\vf^{\varsigma,\th}$ is still a viscosity supersolution. Actually, for any $\vf\in \mathbb{C}^{1,1,1}(\mathbb{R})$  such that $v+\vf^{\varsigma,\th}-\vf$ attains its minimum at $(t,y,v)\in\mathscr{D}$, we have
\[\mathscr{L}[v,\vf-\vf^{\varsigma,\th}](t,y,v)\le 0.\]
Thus, we can see that \[\mathscr{L}[v+\vf^{\varsigma,\th},\vf](s,x,w)\le \mathscr{L}[v,\vf-\vf^{\varsigma,\th}](s,x,w) +\vf_s^{\varsigma,\th}(s,x,w)\le -\frac{\varsigma}{(T+1)^2}. \]
We can also see that at the boundary of $D$, $[v+\vf^{\varsigma,\th}](T,y,w)> 1$ for all $y\ge 0$ and $[v+\vf^{\varsigma,\th}](t,y,v)> 1$ for all $y\ge \eta p(T-t)$.
Denote $v^{\varsigma,\th}:=v+\vf^{\varsigma,\th}$ for simplicity and  call $v^{\varsigma,\th}$ a strict supersolution of HJB equation (\ref{hjb}).  From now on, we  shall argue that $u\le v^{\varsigma,\th}$, which will lead to the desired comparison result as $\lim_{\th \downarrow 0,\varsigma\downarrow 0}v^{\varsigma,\th}=v.$

First, we note that $\lim_{t\rightarrow0} v^{\varsigma,\th}=+\infty.$ Consequently, it suffices to show that
\begin{align}\label{uni3}
u\le v^{\varsigma,\th}\quad \mbox{on}\;\mathscr{D}^*\backslash\{t=0\},
\end{align}
where $\mathscr{D}^*\backslash\{t=0\}:=\{(t,y,v):0<t<T,0\le x<\eta p(T-t), 0\le w\le t\}.$
Suppose (\ref{uni3}) is not true, then there exists a point $Z^*:=(s^*,x^*,w^*)\in \overline{\mathscr{D}^*\backslash \{t=0\}}$ such that
\begin{equation}\nonumber
M:=\sup_{\mathscr{D}^*\backslash\{t=0\}}\left(u(t,y,v)-v^{\varsigma,\th}(t,y,v)\right)=\left(u-v^{\varsigma,\th}\right)(s^*,x^*,w^*)>0.
\end{equation}
Next, denote $\mathscr{D}_0^*:=int \mathscr{D}^*$ and $\mathscr{D}_1^*:=\partial \mathscr{D}^*\backslash\left[\{t=0\}\cup\{x=\eta p(T-t)\}\cup\{t=T\}\right]$. Note that $u-v^{\varsigma,\th}\le 0$ on $t=0, x=\eta p(T-t)$ or $t=T$, thus $(s^*,x^*,w^*)$ can only happen on $\mathscr{D}_0^*\cap\mathscr{D}_1^*$. We consider the following two cases separately.

$Case$ $1$ We assume that $Z^*\in \mathscr{D}_0^*$. We define the function $\P$ on $\overline{\mathscr{D}^*}\times\overline{\mathscr{D}^*}$ by
\begin{equation}
\P(s,x,w,t,y,v)=u(s,x,w)-v^{\varsigma,\th}(t,y,v)-\frac{\k}{2}(s-t)^2-\frac{\k}{2}(x-y)^2-\frac{\k}{2}(w-v)^2.
\end{equation}
Let $M_k:=\max_{\overline{\mathscr{D}^*}\times\overline{\mathscr{D}^*}}\P(s,x,w,t,y,v)$ and $(s_\k, x_\k,w_\k,t_\k,y_\k,v_\k)$ be the maximizer of $\P$. We can see that $M_\k\ge M>0$ for all $\k>1$. Since $\overline{\mathscr{D}^*}\times \overline{\mathscr{D}^*}$ is compact, we can find a subsequence, may assume $(s_\k, x_\k,w_\k,t_\k,y_\k,v_\k)$ itself, such that $(s_\k, x_\k,w_\k,t_\k,y_\k,v_\k)\rightarrow (\hat{s},\hat{x},\hat{w},\hat{t},\hat{y},\hat{v})$. Since $M_\k\ge \P(s^*,x^*,w^*,s^*,x^*,w^*)$, we obtain that
\begin{align}\nonumber
&\frac{\k}{2}(s_\k-t_\k)^2+\frac{\k}{2}(x_\k-y_\k)^2+\frac{\k}{2}(w_\k-v_\k)^2
\\\label{uni4}
\le &u(s_\k,x_\k,w_\k)-v^{\varsigma,\th}(t_\k,y_\k,v_\k)-u(s^*,x^*,w^*)+v^{\varsigma,\th}(s^*,x^*,w^*).
\end{align}
Since  continuous function attains its maximum on any compact set, we see that $\frac{\k}{2}(s_\k-t_\k)^2+\frac{\k}{2}(x_\k-y_\k)^2+\frac{\k}{2}(w_k-v_\k)^2$ is bounded uniformly in $\k$. Thus, we see that $s_\k-t_\k, x_\k-y_\k$ and  $w_\k-v_\k$ convergence to $0$ as $\k\rightarrow 0$, which means, $\hat{s}=\hat{t}, \hat{x}=\hat{y}$ and $\hat{w}=\hat{v}$. Letting $\k\rightarrow\infty$ in (\ref{uni4}), we see that
\begin{align}\nonumber
&\lim_{\k\rightarrow\infty}\left[\frac{\k}{2}(s_\k-t_\k)^2+\frac{\k}{2}(x_\k-y_\k)^2+\frac{\k}{2}(w_\k-v_\k)^2\right]+u(s^*,x^*,w^*)-v^{\varsigma,\th}(s^*,x^*,w^*)\\
&\le u\left(\hat{s},\hat{x},\hat{w}\right)-v^{\varsigma,\th}\left(\hat{s},\hat{x},\hat{w}\right).
\end{align}
By the definition of $(s^*,x^*,w^*)$, we see $\hat{s}=s^*,\hat{x}=x^*,\hat{w}=w^*$. Since $(s^*,x^*,w^*)\in\mathscr{D}_0^*$, we see that for $\k$ large enough, $(s_\k,x_\k,w_\k)\in\mathscr{D}_0^*$ and $(t_\k,y_\k,v_\k)\in\mathscr{D}_0^*$. Define
\begin{align}
&\phi_1(s,x,w):=v^{\varsigma,\th}(t_\k,y_\k,v_\k)+\frac{\k}{2}(s-t_\k)^2+\frac{\k}{2}(x-y_\k)^2+\frac{\k}{2}(w-v_\k)^2,
\\
&\phi_2(t,y,v):=u(s_\k,x_\k,w_\k)-\frac{\k}{2}(s_\k-t)^2-\frac{\k}{2}(x_\k-y)^2-\frac{\k}{2}(w_\k-v)^2.
\end{align}
We observe that $u-\phi_1$ attains its maximum at $(s_\k,x_\k,w_\k)$ and $v^{\varsigma,\th}-\phi_2$ attains its minimum at $(t_\k,y_\k,v_\k)$. By definition, we see that
\begin{align}\nonumber
\max_{\pi\in[0,1]}\bigg\{&p(\pi(1+\eta)-\eta)\k(x_\k-y_\k)+\k(s_\k-t_\k)+\k(w_\k-v_\k)\\&+\lambda(w)\int_0^{\frac{x_\k}{\pi}}u(s_\k,x_\k-\pi y,0)dG(y)
-\lambda(w_\k)u(s_\k,x_\k,w_\k)\bigg\}\ge0,\\\nonumber
\max_{\pi\in[0,1]}\bigg\{&p(\pi(1+\eta)-\eta)\k(x_\k-y_\k)+\k(s_\k-t_\k)+\k(w_\k-v_\k)\\
&+\lambda(v_\k)\int_0^{\frac{y_\k}{\pi}}v^{\varsigma,\th}(t_\k,y_\k-\pi y, 0)dG(y)-\lambda(v_\k)v^{\varsigma,\th}(t_\k,y_\k,v_\k)\le -\frac{\varsigma}{(T+1)^2}.
\end{align}
Letting $\k\rightarrow\infty$, we see \[\lambda(w^*)M\ge \frac{\varsigma}{(T+1)^2}+\lambda(w^*)M,\]
which is a contradiction.

$Case$ $2$ We now consider the case $Z^*\in\mathscr{D}_1^*$. We shall first move this point away the boundary $\mathscr{D}_1^*$ into the interior $\mathscr{D}_0^*$ and then argue as Case $1$. The following construction is a suitable adaption of the construction of Benth et al. \cite{benth}. Since $\mathscr{D}^*$ is a simple polyhedron, it is not hard to see that there exist constants $h_0$, $\xi>0$ and a uniformly continuous map $\gamma:\overline{\mathscr{D}^*}\rightarrow\mathbb{R}^3$ satisfying
\begin{equation}\label{uni5}
\mathscr{N}(X+h\gamma(X),h\xi)\subset \mathscr{D}_0^*\; \;\mbox{for}\; \mbox{all}\; X\in \overline{\mathscr{D}^*}\; \mbox{and}\; h\in(0,h_0],
\end{equation}
where $\mathscr{N}(z,\rho)$ denotes the ball with radius $\rho$ and centre $z$. For any point $X\in\overline{\mathscr{D}^*}$, noticing $\gamma$ is a three-dimensional vector, we write $\gamma(X)=(\gamma_1(X),\gamma_2(X),\gamma_3(X))$. For any $\k>1$ and $0<\delta<1$, define the function $\Phi$ on $\overline{\mathscr{D}^*}\times\overline{\mathscr{D}^*}$ by
\begin{align}\nonumber
\Phi(s,x,w,t,y,v)=&u(s,x,w)-v^{\varsigma,\th}(t,y,v)-(\k(s-t)+\delta \gamma_1(Z^*))^2-(\k(x-y)+\delta\gamma_2(Z^*))^2\\
&-(\k(w-v)+\delta\gamma_3(Z^*))^2-\delta[(s-s^*)^2+(x-x^*)^2+(w-w^*)^2].
\end{align}
Let \[M_\k:=\max_{\overline{\mathscr{D}^*}\times\overline{\mathscr{D}^*}}\Phi(s,x,w,t,y,v).\]
Then we have $M_\k\ge u(s^*,x^*,w^*)-v^{\varsigma,\th}(s^*,x^*,w^*)-\delta^2\gamma(Z^*)^2>0$ for any $\k>1$ and $\delta<\delta_0$, where $\delta_0$ is some fixed small number. Let $(s_\k,x_\k,w_\k,t_\k,y_\k,v_\k)\in\overline{\mathscr{D}^*}\times\overline{\mathscr{D}^*}$ be a maximizer of $\Phi$.
From \[\Phi(s_\k,x_\k,w_\k,t_\k,y_\k,v_\k)\ge \Phi\left(s^*,x^*,w^*,s^*+\frac{\delta}{\k}\gamma_1(Z^*),x^*+\frac{\delta}{\k}\gamma_2(Z^*), w^*+\frac{\delta}{\k}\gamma_3(Z^*)\right),\]
we see that
\begin{align}\nonumber
&|\k(s_\k-t_\k)+\delta\gamma_1(Z^*)|^2+|\k(x_\k-y_\k)+\delta\gamma_2(Z^*)|^2+|\k(w_\k-v_\k)+\delta\gamma_3(Z^*)|^2\\\nonumber
&+\delta((s_\k-s^*)^2+(x_\k-x^*)^2+(w_\k-w^*)^2)\le u(s_\k,x_\k,w_\k)-v^{\varsigma,\th}(t_\k,y_\k,v_\k)\\\label{uni6}
&-v^{\varsigma,\th}(Z^*)-\left(u-v^{\varsigma,\th}\right)(Z^*)+v^{\varsigma,\th}\left(s^*+\frac{\delta}{\k}\gamma_1(Z^*),x^*+\frac{\delta}{\k}\gamma_2(Z^*), w^*+\frac{\delta}{\k}\gamma_3(Z^*)\right).
\end{align}
Since $u$ and $v^{\varsigma,\th}$ are bounded on $\overline{\mathscr{D}^*}$, it follows that $|\k(s_\k-t_\k)|$, $|\k(x_\k-y_\k)|,|\k(w_\k-v_\k)|$ are bounded uniformly in $\k$. Hence, we have $s_\k-t_\k\rightarrow0, x_\k-y_\k\rightarrow 0, w_\k-v_\k\rightarrow 0$ as $\k\rightarrow 0$ and $\lim_{\k\rightarrow\infty}(u(s_\k,x_\k,w_\k)-v^{\varsigma,\th}(t_\k,y_\k,v_\k))\le M.$ Sending $\k\rightarrow\infty$ in (\ref{uni6}) and using the continuity of $u$ and $v^{\varsigma,\th}$, we then conclude that $\k(s_\k-t_\k)+\delta\gamma_1(Z^*)\rightarrow0$, $\k(x_\k-y_\k)+\delta\gamma_2(Z^*)\rightarrow0$, $\k(w_\k-v_\k)+\delta\gamma_3(Z^*)\rightarrow0$, $(s_\k,x_\k,w_\k)\rightarrow Z^*$, $(t_\k,y_\k,v_\k)\rightarrow Z^*$ and $M_\k\rightarrow M$. Therefore, using the uniformly continuity of $\gamma$, $t_\k=s_\k+\frac{\delta}{\k}\gamma_1(s_\k,x_\k,w_\k)+o(\frac{1}{\k})$. Similarly, we see that $y_\k=x_\k+\frac{\delta}{\k}\gamma_2(s_\k,x_\k,w_\k)+o(\frac{1}{\k})$ and $v_\k=w_\k+\frac{\delta}{\k}\gamma_3(s_\k,x_\k,w_\k)+o(\frac{1}{\k})$. We use (\ref{uni5}) to get $(t_\k,y_\k,v_\k)\in\mathscr{D}_0^*$ for $\k$ large enough. Now define
\begin{align}\nonumber
\varphi_1(s,x,w):=&v^{\varsigma,\th}(t_\k,y_\k,v_\k)+(\k(s-t_\k)+\delta\gamma_1(Z^*))^2+(\k(x-y_\k)+\delta\gamma_2(Z^*))^2\\\nonumber
&+(\k(w-v_\k)+\delta\gamma_3(Z^*))^2+\delta(s-s^*)^2+\delta(x-x^*)^2+\delta(w-w^*)^2.\\\nonumber
\varphi_2(t,y,v):=&u(s_\k,x_\k,w_\k)-(\k(s_\k-t)+\delta\gamma_1(Z^*))^2-(\k(x_\k-y)+\delta\gamma_2(Z^*))^2\\\nonumber
&-(\k(w_\k-v)+\delta\gamma_3(Z^*))^2-\delta(s_\k-s^*)^2-\delta(x_\k-x^*)^2-\delta(w_\k-w^*)^2.
\end{align}
Apparently, $u-\varphi_1$ attains its maximum at $(s_\k,x_\k,w_\k)$ and $v^{\varsigma,\th}-\vf_2$ attains its minimum at $(t_\k,y_\k,v_\k)$. By definition, we see that
\begin{align}
\nonumber
\max_{\pi\in[0,1]}\bigg\{&p(\pi(1+\eta)-\eta)\left(2\k\left(\k(x_\k-y_\k)+\delta\gamma_2(Z^*)\right)+2\delta(x_\k-x^*)\right)\\\nonumber
&+2\k(\k(s_\k-t_\k)+\delta\gamma_1(Z^*))+2\k(\k(w_\k-v_\k)+\delta\gamma_3(Z^*))+2\delta(s_\k-s^*)+2\delta(w_\k-w^*)\\\label{uni7}
&+\lambda(w_\k)\int_0^{\frac{x_\k}{\pi}}u(s_\k,x_\k-\pi y,0)dG(y)-\lambda(w_\k)u(s_\k,x_\k,w_\k)\bigg\}\ge 0.\\\nonumber
\max_{\pi\in[0,1]}\bigg\{&2\k p(\pi(1+\eta)-\eta)(\k(x_\k-y_\k)+\delta\gamma_2(Z^*))+2\k(\k(s_\k-t_\k)+\delta\gamma_1(Z^*))-\lambda(v_\k)v^{\varsigma,\th}(t_\k,y_\k,v_\k)\\\label{uni8}
&+2\k(\k(w_\k-v_\k)+\delta\gamma_3(Z^*))+\lambda(v_\k)\int_0^{\frac{y_\k}{\pi}}v^{\varsigma,\th}(t_\k,y_\k-\pi y, 0)dG(y)\bigg\}\le -\frac{\varsigma}{(T+1)^2}.
\end{align}
Combing (\ref{uni7}) and (\ref{uni8}), we send (in that order) $\k\rightarrow\infty$, $\delta\rightarrow 0$ to obtain the desired contradiction
 \[\lambda(w^*)M\ge \frac{\varsigma}{(T+1)^2}+\lambda(w^*)M.\]\
 Now we complete the proof.
 \end{proof}
\section{Concluding Remark}
This paper considers the problem of minimizing ruin probability under the renewal process for the first time.  Since all the theorems and propositions can be proved similarly when we add the investment control variable, we did not consider the effect of investment in this optimization problem.
  The main difficulty of this paper is that there is no discount factor, which makes it tricky to prove the uniqueness of the solution. By constructing strict viscosity solutions and comparing the size of subsolution and the strict supersolution, we show that the value function is the unique viscosity solution of the HJB equation. The uniqueness theory provides theoretical support for possible numerical solutions in future research.

\end{CJK*}
\end{sloppypar}
\end{document}